\documentclass[12pt]{amsart}

\usepackage{amsrefs}
\usepackage{amssymb}
\usepackage{fancyhdr}
\usepackage{bbm}
\usepackage[all]{xy}
\usepackage{xcolor}
\usepackage{hyperref}
\hypersetup{
  colorlinks=true,
  linkcolor=blue,
  citecolor=red
}

\newtheorem{theorem}{Theorem}
\newtheorem*{theorem*}{Theorem}

\newtheorem{proposition}[theorem]{Proposition}
\newtheorem{claim}[theorem]{Claim}

\theoremstyle{definition}

\newtheorem*{definition*}{Definition}
\newtheorem*{lemma*}{Lemma}
\usepackage{graphicx}
\usepackage{pstricks, enumerate, pst-node, pst-text, pst-plot}

\numberwithin{equation}{section}
\numberwithin{theorem}{section}

\newcommand{\R}{\mathbb{R}}

\newcommand{\Z}{\mathbb{Z}}

\newcommand*\rfrac[2]{{}^{#1}\!/_{#2}}

\def\cc{\curvearrowright}

\newcommand{\RNum}[1]{\uppercase\expandafter{\romannumeral #1\relax}}
\newcommand\restr[2]{{% we make the whole thing an ordinary symbol
  \left.\kern-\nulldelimiterspace % automatically resize the bar with \right
  #1 % the function
  \vphantom{\big|} % pretend it's a little taller at normal size
  \right|_{#2} % this is the delimiter
  }}

\begin{document}

\title[]{Thompson's group $F$ is not strongly amenable}

\author[Y.\ Hartman]{Yair Hartman}
%    Address of record for the research reported here

\author[K.\ Juschenko]{Kate Juschenko}
\author[O.\ Tamuz]{Omer Tamuz}
\author[P.\ Vahidi Ferdowsi]{Pooya Vahidi Ferdowsi}
%    Address of record for the research reported here
\address[Y.\ Hartman, K.\ Juschenko]{Northwestern University}
\address[O.\ Tamuz, P.\ Vahidi Ferdowsi]{California Institute of Technology}

%    Information for third author

%\thanks{}
\date{\today}

\begin{abstract}
  We show that Thompson's group $F$ has a topological action on a
  compact metric space that is proximal and has no fixed points.
\end{abstract}

\maketitle
%\tableofcontents
\section{Introduction}
In his book ``Proximal Flows''~\cite[Section~\RNum{2}.3, p.\ 19]{glasner1976proximal} Glasner
defines the notion of a {\em strongly amenable group}: A group is
strongly amenable if each of its proximal actions on a compact
space has a fixed point.  A continuous action $G \cc X$ of a
topological group on a compact Hausdorff space is proximal if for
every $x, y \in X$ there exists a net $\{g_n\}$ of elements of $G$
such that $\lim_n g_n x = \lim_n g_n y$.

Glasner shows that virtually nilpotent groups are strongly amenable
and that non-amenable groups are not strongly amenable. He also gives
examples of amenable --- in fact, solvable --- groups that are not
strongly amenable. Glasner and
Weiss~\cite{glasner2002minimal} construct proximal minimal actions of
the group of permutations of the integers, and Glasner constructs proximal flows of Lie groups~\cite{glasner1983proximal}. To the best of our knowledge there are no other such examples known. Furthermore, there are no other known examples of minimal proximal actions that are not also {\em strongly
  proximal}. An action $G \cc X$ is strongly proximal if the orbit
closure of every Borel probability measure on $G$ contains a point
mass measure. This notion, as well as that of the related Furstenberg
boundary~\cites{furstenberg1963poisson, furstenberg1973boundary,
  furman2003minimal}, have been the object of a much larger research
effort, in particular because a group is amenable if and only if all
of its strongly proximal actions on compact spaces have fixed points.

Richard Thompson's group $F$ has been alternatively ``proved'' to be
amenable and non-amenable (see, e.g.,~\cite{cannon2011thompson}), and
the question of its amenability is currently unresolved. In this paper
we pursue the less ambitious goal of showing that is it not strongly
amenable, and do so by directly constructing a proximal action that
has no fixed points. This action does admit an invariant measure, and
thus does not provide any information about the amenability of $F$. It
is a new example of a proximal action which is not strongly proximal.

\vspace{0.3in}

The authors would like to thank Eli Glasner and Benjamin Weiss for
enlightening and encouraging conversations.

\section{Proofs}
Let $F$ denote Thompson's group $F$. In the representation of $F$ as a
group of piecewise linear transformations of $\R$ (see,
e.g.,~\cite[Section 2.C]{kaimanovich2016thompson}), it is generated by $a$ and $b$ which are
given by
\begin{align*}
a(x) &= x-1\\
b(x) &= \begin{cases}
x& x \leq 0\\
x/2& 0 \leq x \leq 2\\
x-1& 2 \leq x.
\end{cases}
\end{align*}

The set of dyadic rationals $\Gamma =\Z[\frac{1}{2}]$ is the orbit of
$0$. The Schreier graph of the action $G \cc \Gamma$ with respect to
the generating set $\{a,b\}$ is shown in Figure~\ref{fig:schreier}
(see~\cite[Section 5.A, Figure 6]{kaimanovich2016thompson}). The solid
lines denote the $a$ action and the dotted lines denote the $b$
action; self-loops (i.e., points stabilized by a generator) are
omitted. This graph consists of a tree-like structure (the blue and
white nodes) with infinite chains attached to each node (the red
nodes).
% We call each chain of red nodes a {\em hair}.%
\begin{figure}[ht]
  \centering
  \includegraphics[scale=0.6]{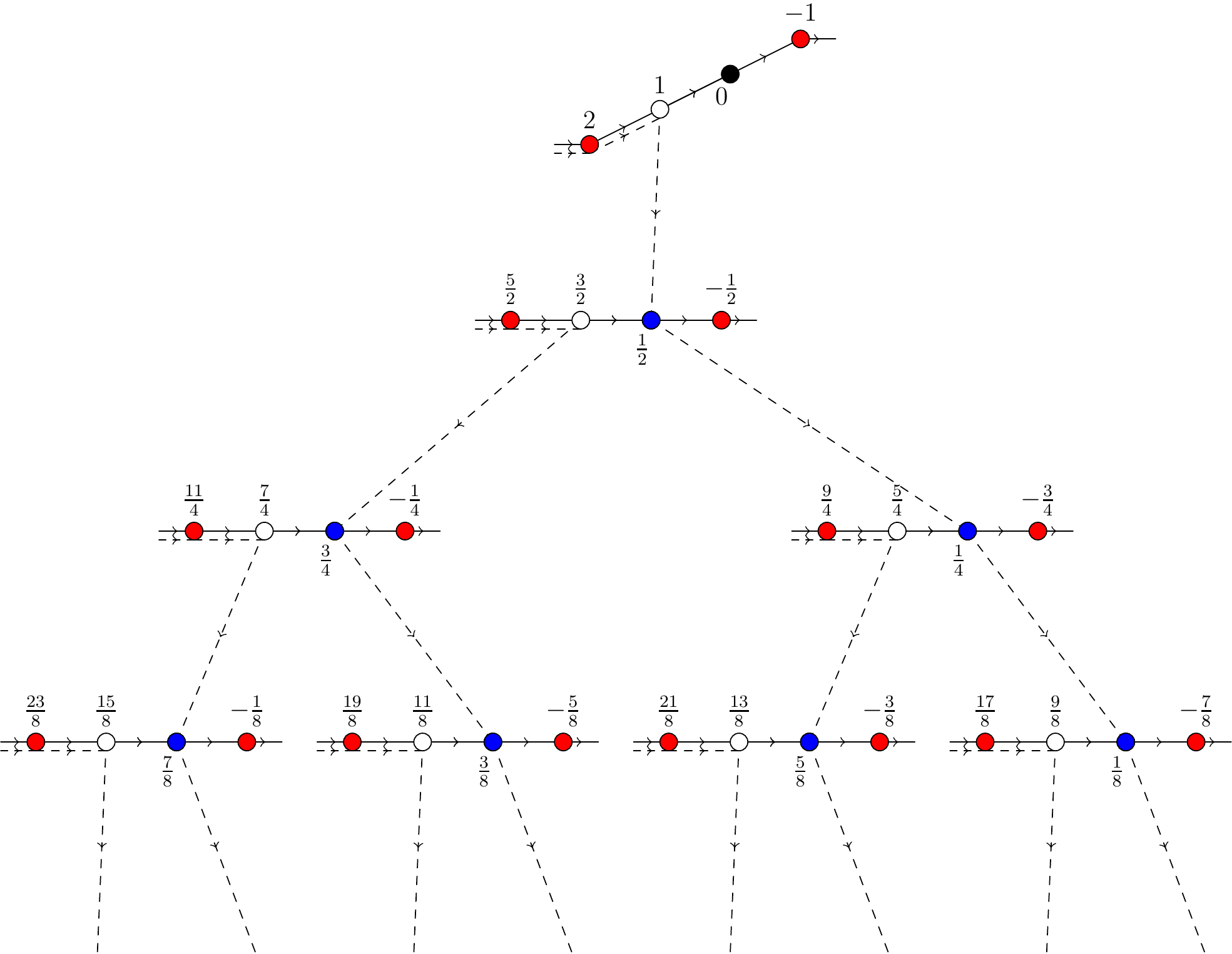}
  
  \caption{\label{fig:schreier}The action of $F$ on $\Gamma$.}
\end{figure}

Equipped with the product topology, $\{-1,1\}^\Gamma$ is a compact
space on which $F$ acts continuously by shifts:
\begin{align}
	\label{shift-action}
  	[f x](\gamma) = x(f^{-1}\gamma).
\end{align}

\begin{proposition}
  \label{prop:pre_proximal}
  Let $c_{-1}, c_{+1} \in \{-1,1\}^{\Gamma}$ be the constant
  functions. Then for any $x \in \{-1,1\}^{\Gamma}$ it holds that at
  least one of $c_{-1},c_{+1}$ is in the orbit closure
  $\overline{F x}$.
\end{proposition}
\begin{proof}
  It is known that the action $F \cc \Gamma$ is highly-transitive
  (Lemma 4.2 in ~\cite{cannon1994notes}), i.e. for every finite $V, W \subset \Gamma$ of the same size there exists
  a $f \in F$ such that $f(V)=W$.  Let $x\in \{-1,1\}^{\Gamma}$. There
  is at least one of -1 and 1, say $\alpha$, for which we have
  infinitely many $\gamma \in \Gamma$ with $x(\gamma)=\alpha$. Given a finite $W \subset \Gamma$ choose a $V \subset \Gamma$ of the same size and such that $x(\gamma) = \alpha$ for all $\gamma \in V$. Then there is some $f \in F$ with $f(V) = W$, and so $f x$ takes the value $\alpha$ on $W$. Since $W$ is arbitrary we have that $c_\alpha$ is in the orbit closure of $x$.
\end{proof}

Given $x_1,x_2 \in \{-1,1\}^{\Gamma}$, let $d$ be their pointwise product,
given by $d(\gamma) = x_1(\gamma) \cdot x_2(\gamma)$. By
Proposition~\ref{prop:pre_proximal} there exists a sequence $\{f_n\}$
of elements in $F$ such that either $\lim_n f_n d = c_{+1}$ or
$\lim_n f_n d = c_{-1}$. In the first case
$\lim_n f_n x_1 = \lim_n f_n x_2$, while in the second case
$\lim_n f_n x_1 = -\lim_n f_n x_2$, and so this action resembles a
proximal action. In fact, by identifying each
$x \in \{-1,1\}^{\Gamma}$ with $-x$ one attains a proximal action, and
indeed we do this below. However, this action has a fixed point ---
the constant functions --- and therefore does not suffice to prove our
result. We spend the remainder of this paper in deriving a new action
from this one. The new action retains proximality but does not have
fixed points.

Consider the path
$(\rfrac{1}{2},
\rfrac{1}{4},\rfrac{1}{8},\ldots,\rfrac{1}{2^n},\ldots)$
in the Schreier graph of $\Gamma$ (Figure~\ref{fig:schreier}); it
starts in the top blue node and follows the dotted edges through the
blue nodes on the rightmost branch of the tree. The pointed
Gromov-Hausdorff limit of this sequence of rooted graphs\footnote{The
  limit of a sequence of rooted graphs $(G_n,v_n)$ is a rooted graph
  $(G,v)$ if each ball of radius $r$ around $v_n$ in $G_n$ is, for $n$
  large enough, isomorphic to the ball of radius $r$ around $v$ in $G$
  (see, e.g.,~\cite[p.\ 1460]{aldous2007processes}).} is given in
Figure~\ref{fig:schreier2}, and hence is also a Schreier graph of some
transitive $F$-action $F \cc F/K$.  In terms of the topology on the
space $\mathrm{Sub}_F \subset \{0,1\}^F$ of the subgroups of $F$, the
subgroup $K$ is the limit of the subgroups $K_n$, where $K_n$ is the
stabilizer of $\rfrac{1}{2^n}$. It is easy to verify that $K$ is the
subgroup of $F$ consisting of the transformations that stabilize $0$
and have right derivative $1$ at $0$ (although this fact will not be
important). Let $\Lambda = F/K$.

\begin{figure}[ht]
  \centering
  \includegraphics[scale=0.6]{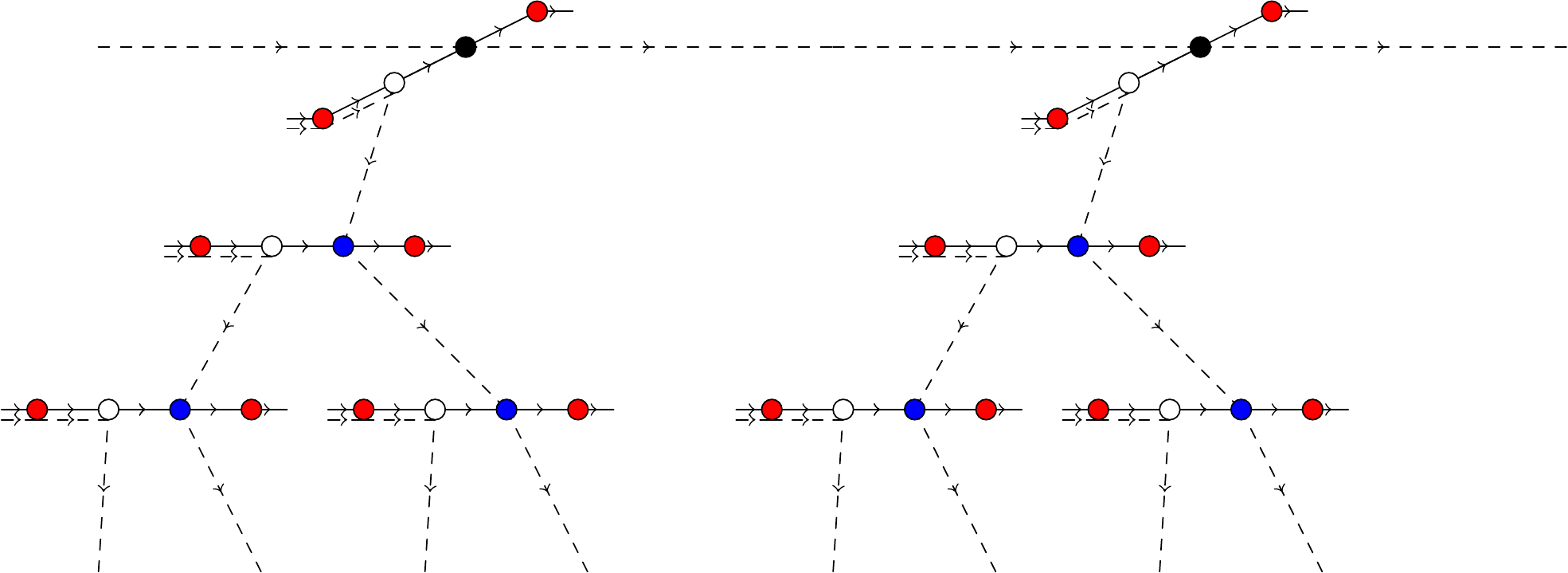}
  
  \caption{\label{fig:schreier2}The action of $F$ on $\Lambda$.}
\end{figure}

We can naturally identify with $\Z$ the chain black nodes at the top
of $\Lambda$ (see Figure~\ref{fig:schreier2}). Let $\Lambda'$ be the
subgraph of $\Lambda$ in which the dotted edges connecting the black
nodes have been removed.  Given a black node $n \in \Z$, denote by
$T_n$ the connected component of $n$ in $\Lambda'$; this includes the
black node $n$, the chain that can be reached from it using solid
edges, and the entire tree that hangs from it. Each graph $T_n$ is
isomorphic to the Schreier graph of $\Gamma$, and so the graph
$\Lambda$ is a covering graph of $\Gamma$ (in the category of Schreier
graphs).  Let
\begin{align*}
\Psi \colon \Lambda \to \Gamma
\end{align*}
be the covering map. That is, $\Psi$ is a graph isomorphism when restricted to each $T_n$, with the black nodes in $\Lambda$ mapped to the black node $0 \in \Gamma$.

Using the map $\Psi$ we give names to the nodes in $\Lambda$.  Denote
the nodes in $T_0$ as $\{(0, \gamma) \,:\, \gamma \in \Gamma\}$ so
that $\Psi(0,\gamma) = \gamma$. Likewise, in each $T_n$ denote by
$(n,\gamma)$ the unique node in $T_n$ that $\Psi$ maps to
$\gamma$. Hence we identify $\Lambda$ with
\begin{align*}
\mathbb{Z} \times \Gamma = \{(n, \gamma)\,:\, n \in \Z, \gamma \in \Gamma\}
\end{align*}
and the $F$-action is given by
\begin{align}
\label{a-action-on-Lambda}
a (n,\gamma) &= (n, a \gamma)\\
\label{b-action-on-Lambda}
b (n,\gamma) &= \begin{cases}
(n, b \gamma)&\mbox{if }\gamma \neq 0\\
(n+1, 0)&\mbox{if }\gamma= 0
\end{cases}
\end{align}

Equip $\{-1,1\}^\Lambda$ with the product topology to get a compact space. As usual, the $F$-action on $\Lambda$ (given explicitly in ~\ref{a-action-on-Lambda} and ~\ref{b-action-on-Lambda}) defines a continuous action on $\{-1,1\}^\Lambda$.

Consider $\pi:\{-1,1\}^\Gamma \to \{-1,1\}^\Lambda$, given by $\pi(x)(n, \gamma) = (-1)^n x(\gamma)$. Let $Y = \pi(\{-1,1\}^\Gamma) \subseteq \{-1,1\}^\Lambda$.

\begin{claim}
  \label{clm:compact-and-invariant}
  $Y$ is compact and $F$-invariant.
\end{claim}
\begin{proof}
$\pi$ is injective and continuous, so $Y = \pi(\{-1,1\}^\Gamma) \subseteq \{-1,1\}^\Lambda$ is compact and isomorphic to $\{-1,1\}^\Gamma$. Moreover, $Y$ is invariant to the action of $F$, because $a^{\pm 1}\pi(x) = \pi (a^{\pm 1}x)$ and $b^{\pm 1}\pi(x) = \pi(b^{\pm}\bar{x})$ where $\bar{x}(\gamma) = \begin{cases}
x(\gamma)&\mbox{if }\gamma \neq 0\\
-x(\gamma)&\mbox{if } \gamma = 0
\end{cases}$.
\end{proof}

The last $F$-space we define is $Z$, the set of pairs of mirror image configurations in $Y$:
\begin{align}
\label{the-space-Z}
  Z = \left\{\{y, -y\}\,:\,y\in Y \right\}.
\end{align}

Now it is clear that equipped with the quotient topology, $Z$ is a
compact and Hausdorff $F$-space. Furthermore, we now observe that $Z$ admits an
invariant measure.  Consider the i.i.d.\ Bernoulli $1/2$ measure on
$\{-1,1\}^\Gamma$, i.e. the unique Borel measure on $\{-1,1\}^\Gamma$, for which
\begin{align*}
X_\gamma \colon & \{-1,1\}^\Gamma \to \{0, 1\},\quad x\mapsto \frac{x(\gamma)+1}{2}
\end{align*} are independent Bernoulli $1/2$ random variables for all $\gamma \in \Gamma$. Clearly, it is an invariant measure and hence it
is pushed forward to an invariant measure on $Y$, and then on $Z$. In particular, this
shows that $Z$ is not strongly proximal.

\begin{claim}
  \label{clm:no-fixed-points}
  The action $F \cc Z$ does not have any fixed points.
\end{claim}
\begin{proof}
Pick $\hat{y} = \{y, -y\}\in Z$.
We have $[by](0, -1) = y(0, -1) \neq -y(0, -1)$, so $by\neq -y$. Similarly, $[b y](0, 0) = y(-1, 0) = -y(0, 0) \neq y(0, 0)$, and so $by \neq y$. Hence $b\hat{y}\neq \hat{y}$.
\end{proof}

\begin{proposition}
  \label{thm:proximal}
  The action $F \cc Z$ is proximal.
\end{proposition}

\begin{proof}
Let $\hat{y_1}=\{y_1, -y_1\}$ and $\hat{y_2}=\{y_2,-y_2\}$ be two points in $Z$, and let $y_i=\pi(x_i)$.

Let $x_1 \cdot x_2$ denote the pointwise product of $x_1$ and $x_2$. Now by Proposition~\ref{prop:pre_proximal} there is a sequence of elements $\{f_n\}_n$ in $F$ such that $\{f_n (x_1 \cdot x_2)\}_n$ tends to either $c_{-1}$ or $c_{+1}$ in $\{-1,1\}^\Gamma$. Since $Y$ is compact, we may assume that $\{f_n y_1\}_n$ and $\{f_n y_2\}_n$ have limits, by descending to a subsequence if necessary.

It is straightforward to check that $f_n y_1 \cdot f_n y_2 = f_n\pi(x_1)\cdot f_n\pi(x_2)=\pi(f_n x_1) \cdot \pi(f_n x_2)$. So:
\begin{align*}
[f_n y_1 \cdot f_n y_2](n,\gamma) &= [\pi(f_n x_1) \cdot \pi(f_n x_2)](n, \gamma)\\
&= (-1)^{2n}\;[f_n x_1](\gamma)\;[f_n x_2](\gamma)\\
&=[f_n x_1 \cdot f_n x_2](\gamma) = [f_n (x_1 \cdot x_2)](\gamma)
\end{align*}
So $\lim_n f_n y_1 = \pm \lim_n f_n y_2$, which implies $\lim_n f_n \hat{y_1} = \lim_n f_n \hat{y_2}$.
\end{proof}

\begin{theorem}
 Thompson's group $F$ is not strongly amenable.
\end{theorem}
\begin{proof}
Since the space $Z$ we constructed above is proximal (Proposition~\ref{thm:proximal}), and has no fixed points (Claim~\ref{clm:no-fixed-points}), we conclude that $F$ has a proximal action with no fixed points, so $F$ is not strongly amenable.
\end{proof}

\bibliography{thompson}
\end{document}